\documentclass[12pt]{amsart}
\usepackage{amsmath,amsfonts,amsthm,amssymb}
\usepackage[all]{xy}
\usepackage{textcomp}

\newtheorem{theorem}{Theorem}[section]
\newtheorem{proposition}[theorem]{Proposition}
\newtheorem{corollary}[theorem]{Corollary}
\newtheorem{lemma}[theorem]{Lemma}

\theoremstyle{definition}
\newtheorem{definition}[theorem]{Definition}

\newtheorem{example}[theorem]{Example}
\newtheorem{remark}[theorem]{Remark}

\DeclareMathOperator{\B}{B}
\DeclareMathOperator{\spec}{Spec}
\DeclareMathOperator{\Vect}{Vect}
\DeclareMathOperator{\Aut}{Aut}
\DeclareMathOperator{\BAut}{\mathbf{Aut}}
\DeclareMathOperator{\EF}{EFVect}
\DeclareMathOperator{\Gr}{Gr}
\DeclareMathOperator{\Bd}{Bd}
\DeclareMathOperator{\bd}{bd}
\DeclareMathOperator{\Ob}{Ob}

\DeclareMathOperator{\R}{R}

\CompileMatrices

\begin{document}

\title{The Nori fundamental gerbe of tame stacks}

\author{Indranil Biswas}

\address{School of Mathematics, Tata Institute of Fundamental Research,
Homi Bhabha Road, Bombay 400005, India}

\email{indranil@math.tifr.res.in}

\author{Niels Borne}

\address{Universit\'e Lille 1, Cit\'e scientifique
U.M.R. CNRS 8524, U.F.R. de Math\'ematiques
59 655 Villeneuve d'Ascq C\'edex, France}

\email{Niels.Borne@math.univ-lille1.fr}


 
\begin{abstract}
Given an algebraic stack, we compare its Nori fundamental group with that of its
coarse moduli space. We also study conditions under which the stack can be
uniformized by an algebraic space.
\end{abstract}

\maketitle

\section{Introduction}

The aim here is to show that the results of \cite{noohi:fund_group_alg_stack} 
concerning the \'etale fundamental group of algebraic stacks also hold for the Nori 
fundamental group.

Let us start by recalling Noohi's approach. Given a connected algebraic stack 
$\mathfrak X$, and a geometric point $x\,:\,\spec \Omega \,\longrightarrow\, 
\mathfrak X$, Noohi generalizes the definition of Grothendieck's \'etale 
fundamental group to get a profinite group $\pi_1(\mathfrak X,x)$ which classifies 
finite \'etale \emph{representable} morphisms (coverings) to $\mathfrak X$. He then 
highlights a new feature specific to the stacky situation: for each geometric point 
$x$, there is a morphism $\omega_x\,:\,\Aut x \,\longrightarrow\, \pi_1(\mathfrak 
X,x)$.

Noohi first studies the situation where $\mathfrak X$ admits a moduli space $Y$, 
and proceeds to show that if $N$ is the closed normal subgroup of $\pi_1(\mathfrak 
X,x)$ generated by 
the images of $\omega_x$, for $x$ varying in all geometric points, then 
$$\frac{\pi_1(\mathfrak X,x)}{N} \,\simeq\, \pi_1(Y,y)\; .$$ Noohi turns next to 
the issue of uniformizing algebraic stacks: he defines a Noetherian algebraic 
stack $\mathfrak X$ as uniformizable if it admits a covering, in the above sense, 
that is an algebraic space. His main result is that this happens precisely when 
$\mathfrak X$ is a Deligne--Mumford stack and for any geometric point $x$, the 
morphism $\omega_x$ is injective.

For our purpose, it turns out to be more convenient to use the Nori 
fundamental gerbe defined in \cite{bv:nori_gerbe}. For simplicity, we will assume in the 
rest of this introduction that $\mathfrak X$ is a proper, geometrically connected and 
reduced algebraic stack over a field $k$, so that a fundamental gerbe $$\mathfrak X
\,\longrightarrow\, 
\pi_{\mathfrak X/k}$$ exists, and has a Tannakian interpretation. An essential role is 
played by residual gerbes at closed points $x$ of $\mathfrak X$, denoted by $\mathcal 
G_x\,\longrightarrow\,\mathfrak X$.

Let us first describe the content of Section \ref{sec:generators}. We assume that 
$\mathfrak X$ admits a good moduli space $Y$ in the sense of Alper \cite{alper:good_moduli} (this is the case, for instance when $\mathfrak X$ is tame 
as defined in \cite{aov:tame_stacks}) and proceed to compare the fundamental gerbes $\pi_{\mathfrak 
X/k}$ and $\pi_{Y/k}$. We use a result of Alper 
relating vector bundles on $ \mathfrak 
X$ and on $Y$ to show that the morphism $\pi_{\mathfrak X/k}\,\longrightarrow\,
\pi_{Y/k}$ is universal with respect to the property that all composites
$$
\mathcal G_x\,\longrightarrow\, \mathfrak X\,\longrightarrow\, \pi_{\mathfrak X/k}\,
\longrightarrow\, \pi_{Y/k}
$$
are trivial, in a natural sense (see Corollary 
\ref{cor:gen_fg_moduli}). Using Alper's theorem again, we also prove (Proposition 
\ref{cor:tors-mod-space}) that a given $G$--torsor $\mathfrak X'\,\longrightarrow\,
\mathfrak X$ is the 
pullback of a $G$--torsor on $Y$ via the morphism $\mathfrak X\,\longrightarrow\,
Y$ if and only if it is isovariant.

In Section \ref{sec:unif}, we work with stacks with finite inertia (again, tame stacks 
are examples). We say that such a stack $\mathfrak X$ over a field $k$ is 
Nori-uniformizable if there exists a finite $G$--torsor $X'
\,\longrightarrow\, \mathfrak X$, where the total space 
$X'$ is an algebraic space. 
Our main technical result (Proposition \ref{unif_res}) states that $\mathfrak X$ is 
Nori-uniformizable if and only if all composite morphisms
$$\mathcal G_x\,\longrightarrow\,\mathfrak X\,\longrightarrow\, \pi_{\mathfrak X/k}$$
are representable. This is a ``continuous'' version of Noohi's main theorem, 
and this formulation also demonstrates how convenient it is to use Nori's fundamental gerbe instead of Nori's fundamental group scheme. Our
main result, Theorem \ref{char_unif}, is a Tannakian translation of
Proposition \ref{unif_res} that gives a characterization of Nori-uniformizability in
terms of restriction of essentially vector bundles on $\mathfrak X$ along all
morphisms $\mathcal G_x\,\longrightarrow\, \mathfrak X$. It states, morally, that
$\mathfrak X$ is
Nori-uniformizable if and only if for all $x$, any representation of $\mathcal G_x$
comes from an essentially finite vector bundle on $\mathfrak X$. We hope to be able to
apply this result to certain orbifolds (called stack of roots) to relate
Nori-uniformizability to parabolic bundles.

We conclude this introduction by pointing out that no properness assumption is 
needed to prove Proposition \ref{unif_res}, while it is essential in our proofs of 
Corollary \ref{cor:gen_fg_moduli} and Proposition \ref{cor:tors-mod-space}. Since 
Noohi's counterparts hold for any algebraic stack, it is an interesting question if 
it is possible to remove this hypothesis, but we have no idea of a proof avoiding 
Tannaka duality at the moment.

\section{Preliminaries}

We work over a base field $k$, and denote $S\,=\,
\spec k$. We will mainly be interested in the case where the characteristic $p$ of $k$ is positive.

Concerning Nori fundamental gerbes, we use the terminology introduced in \cite{bv:nori_gerbe}. Let us sum up the conventions used and refer to \cite{bv:nori_gerbe} for more information.

A Tannakian gerbe (see \cite[\S 3]{bv:nori_gerbe}) over $S$ is a fpqc gerbe with affine diagonal and an affine chart. For such gerbes, Tannaka duality holds: our reference is \cite[III \S 3]{rivano:cat_tan} in the corrected formulation given in \cite{del:cat_tak}. The Nori fundamental gerbe is a Tannakian gerbe.

It is even an inverse limit of finite gerbes. Recall from \cite[\S 
4]{bv:nori_gerbe} that a finite gerbe is a fppf gerbe with finite diagonal and a 
finite flat chart. By Artin's theorem, this is an algebraic stack.

Given an algebraic stack $\mathfrak X/S$, we say that $\mathfrak X$ is inflexible if any morphism to a finite stack factors through a gerbe (see \cite[Definition 5.3]{bv:nori_gerbe}). This condition is equivalent to the existence of a Nori fundamental gerbe $\pi_{\mathfrak{X}/S}$, i.e., a morphism to a profinite gerbe $\mathfrak X \,\longrightarrow\, \pi_{\mathfrak{X}/S}$ which is universal. It is realized for instance when $\mathfrak X/S$ is of finite type, geometrically connected and geometrically reduced.

We now turn to the Tannakian interpretation of the Nori fundamental gerbe. Recall 
that according to Nori a vector bundle $\mathcal E$ is called finite is there
is a non trivial relation between its tensor powers (see \cite{nor_phd}). Formally, 
this means that there are two distinct polynomials $f\, ,g\,\in\, \mathbb N[t]$ such that 
$f(\mathcal E)\,\simeq\, g(\mathcal E)$, when we replace $+$ by $\oplus$ and $\cdot$ 
by $\otimes$ when we evaluate a polynomial at a vector bundle. We adopt the 
definition of an essentially finite vector bundle given in \cite{bv:nori_gerbe}, 
so essentially finite vector bundles are precisely the kernels of morphisms 
between two finite vector bundles.

We will say that an algebraic stack $\mathfrak X/S$ is pseudo-proper if for any vector bundle $\mathcal E$ on $\mathfrak X$, the space of global sections $\Gamma(X,\mathcal E)$ is finite dimensional over $k$ (see \cite[Definition 7.1]{bv:nori_gerbe} for the precise definition). If $\mathfrak X/S$ is inflexible and pseudo-proper, pull-back along $\mathfrak X \,\longrightarrow\, \pi_{\mathfrak{X}/S}$ identifies representations of $\pi_{\mathfrak{X}/S}$ with essentially finite vector bundles on $\mathfrak X$ (\cite[Theorem 7.9]{bv:nori_gerbe}), thus we get in this situation a Tannakian interpretation of the Nori fundamental gerbe.

Our main reference for stacks is the \emph{Stacks Project} \cite{stacks-project}. If $x$ is a 
point of an algebraic stack $\mathfrak X$, we will write $\mathcal G_x$ for the residual gerbe at $x$, see
\cite[Tag 06ML]{stacks-project}. This is a reduced stack with a single 
point, and there is a canonical monomorphism $\mathcal G_x\,\longrightarrow\,
 \mathfrak X$ mapping this unique point to $x$. By closed point of a stack, we mean as usual a geometric point with closed image.

\section{Generators of the Nori fundamental gerbe}
\label{sec:generators}

In this section, we will deal with algebraic stacks $\mathfrak X$ with a \emph{good 
moduli space} $$\varphi\,:\,\mathfrak X \,\longrightarrow \,Y$$ in the sense of \cite{alper:good_moduli}.

\subsection{Characterization of essentially finite vector bundles coming from the moduli space}

\begin{proposition}
	\label{prop:ef_good_mod}
Assume $\mathfrak X$ is a locally Noetherian algebraic stack with good moduli
space $$\varphi\,:\,\mathfrak X \,\longrightarrow\, Y\, .$$ The functors $\varphi^*$ and
$\varphi_*$ induce an
equivalence of categories between the category of essentially finite vector bundles $\mathcal F$ on $Y$ and the full subcategory of essentially finite vector bundles $\mathcal E$ on $\mathfrak X$ satisfying the condition that for any closed point $x$, the restriction $\mathcal E_{|\mathcal G_x}$ is trivial.
\end{proposition}

\begin{proof}
The fact that the same result holds for vector bundles is proved in \cite[Theorem 
10.3]{alper:good_moduli}. Since $\varphi^*$ is compatible with tensor products, so is the 
inverse equivalence $\varphi_*$, hence the equivalence holds for finite vector bundles. By 
definition of a good moduli space, the functor $\varphi_*$ is exact, and so is the inverse equivalence 
$\varphi^*$. Since an essentially finite vector bundle is by definition the kernel of a 
morphism between two finite vector bundles, the equivalence holds for essentially finite 
vector bundles as well.
\end{proof}

\begin{remark}
It is unclear if the statement holds for \emph{adequate} moduli spaces in the sense
of \cite{alper:adequate_moduli}. For vector bundles it is false according
to \cite[Example 5.6.1]{alper:adequate_moduli}.  
\end{remark}

\subsection{Fundamental gerbe of the moduli space}

We now use Tannaka duality to translate Proposition \ref{prop:ef_good_mod} in 
terms of fundamental gerbes. Morally, $\pi_{Y/S}$ is the quotient gerbe obtained 
from $\pi_{\mathfrak X/S}$ after dividing by the ``normal sub-gerbe generated by 
the images of $\mathcal G_x \,\longrightarrow\, \pi_{X/S}$''. We must be careful 
since the $\mathcal G_x$ are not necessarily defined over the base field $k$, but 
only over the extension $k(\varphi(x))$. The precise definitions are as follows.

\begin{definition}
Let $k'/k$ be a field extension, and $\mathcal G$ (respectively, $\mathcal G'$) a gerbe
over $S\,=\,\spec k$, (respectively, $S'\,=\,\spec k'$). A $S'\,\longrightarrow\,
S$--morphism $\mathcal G'\,\longrightarrow\, \mathcal G$ is \emph{trivial} if
there is a morphism $S'\,\longrightarrow\, \mathcal G$ (shown below in dotted arrow) 
\[
	\xymatrix{
	\mathcal G'\ar[r]\ar[d] &\mathcal G \ar[d]\\
	S'\ar@{.>}[ur] \ar[r]& S
	}
\]
making both triangles commute.
\end{definition}
If the gerbes are Tannakian, this means by duality that the pullback functor $$\Vect \mathcal G
\,\longrightarrow\, \Vect \mathcal G'$$ sends any object to a trivial one.

\begin{definition}
	Let $(k_i/k)_{i\in I}$ be a family of field extensions, and for each $i\,\in\, I$, we
are given a
$k_i/k$--morphism $$\alpha_i\,:\,\mathcal G_i
\,\longrightarrow\,
\mathcal G$$ from a $k_i$--gerbe $\mathcal G_i$ to a fixed $k$--gerbe $\mathcal G$. We say that a morphism of $k$--gerbes $\mathcal G\,\longrightarrow\,
\widetilde{\mathcal G}$ is a quotient by the ``normal sub-gerbe generated by the images of the $\alpha_i$'s'' if all composite morphisms $\mathcal G_i\,\longrightarrow\,
 \widetilde{\mathcal G}$ are trivial, and $\mathcal G \,\longrightarrow\,
\widetilde{\mathcal G}$ is universal for this property.
\end{definition}

It is clear that the quotient gerbe, if it exists, is unique. The existence follows, 
when all gerbes are Tannakian, from duality; indeed, it is enough to define 
$\widetilde{\mathcal G}$ as the Tannaka dual of the full subcategory of $\Vect \mathcal 
G$ generated by objects that are trivialized by all pullback functors $\alpha_i^*\,:\, \Vect 
\mathcal G \,\longrightarrow\,\Vect \mathcal G_i$.

\begin{remark}
	Even if the quotient makes sense, the ``normal sub-gerbe generated by the
images of the $\alpha_i$'s'' doesn't always exist, and even if it exists, it is not
uniquely defined (see \cite{milne_quot}).
\end{remark}

\begin{corollary}
\label{cor:gen_fg_moduli}
Let $\mathfrak X$ be a locally Noetherian algebraic stack with good moduli space
$\varphi\,:\,
\mathfrak X \,\longrightarrow\, Y$. Assume that both $\mathfrak X$ and $Y$ are inflexible (i.e.,
admit fundamental gerbes) and are pseudo-proper. Then the fundamental gerbe $\pi_{Y/S}$ is the quotient of $\pi_{\mathfrak X /S}$ by the normal sub-gerbe generated by the images of the morphisms $\mathcal G_x \,\longrightarrow\, \pi_{\mathfrak X /S}$.
\end{corollary}

\begin{proof}
This follows from the Tannakian interpretation of fundamental gerbes (see 
\cite[Theorem 7.9]{bv:nori_gerbe}) and Proposition \ref{prop:ef_good_mod} by duality.
\end{proof}

\begin{example}\label{ex_gen}
We assume that $k$ is of positive characteristic $p$.
Consider the standard action of $\mu_p\,\subset\, \mathbb G_m$ on ${\mathbb P}^1$ and
put $$\mathfrak X \,=\, [ {\mathbb P}^1 / \mu_p]\, .$$ Then
$\varphi\,:\,\mathfrak X\,\longrightarrow
\,{\mathbb P}^1$ is a good moduli space (because $\mathfrak X$ is in
fact tame, see \cite{aov:tame_stacks}), and Corollary \ref{cor:gen_fg_moduli}
applies: $\pi_{\mathfrak X/S}$ is generated by $\mathcal G_0$ and
$\mathcal G_\infty$. In fact, it is easy to show directly
that $\pi_{\mathfrak X/S}\,=\,\B \mu_p$.	
\end{example}

\subsection{Characterization of torsors coming from the moduli space}

\begin{definition}
	Let $f\,:\,\mathfrak X \,\longrightarrow\,
 \mathcal G$ be a morphism from an algebraic stack to a finite gerbe. We say that
\emph{$f$ is trivial on inertia} if the morphism $\mathcal I_{\mathfrak X}\,\longrightarrow\, 
\mathcal I_{\mathcal G}$ induced by $f$ factors through the unit morphism
$\mathcal G\,\longrightarrow\, \mathcal I_{\mathcal G}$.
\end{definition}

Clearly, $f$ is trivial on inertia if and only if for any section $\sigma \,:\, T\,\longrightarrow\, \mathfrak X$, the induced morphism of $T$--group spaces $\BAut_T \sigma \,\longrightarrow\, \BAut_T f(\sigma)$ is trivial.

The following corollary of Proposition \ref{prop:ef_good_mod} provides us with an
interpretation of $$\mathfrak X \,\longrightarrow\,
 \pi_{Y/S}$$ as the limit over all morphisms $\mathfrak X\,\longrightarrow\,
\mathcal G$ that are trivial on inertia.

\begin{corollary}\label{cor:morph_fact_moduli}
With the hypothesis of Corollary \ref{cor:gen_fg_moduli}, a given morphism to a finite
gerbe $$f\,:\,\mathfrak X \,\longrightarrow\,
\mathcal G$$ factors through $\mathfrak X\,\longrightarrow\,
 Y$ if and only if $f$ is trivial on inertia. 
\end{corollary}

\textit{Proof.}~ The ``only if" direction is obvious, thus we assume $f$ is trivial on
inertia. By  Tannaka duality, we must prove that the functor $$f^*\,:\,\Vect \mathcal G 
\,\longrightarrow\,
 \EF \mathfrak X$$ factors through $\EF Y$, or in other words, according to Proposition
\ref{prop:ef_good_mod}, that for any representation $V$ of $\mathcal G$, and any closed
point $x$ of $\mathfrak X$, the restriction $f^*V_{|\mathcal G_x}$ is trivial. This
follows from the fact that for any geometric point $x\,:\,\spec \Omega\,\longrightarrow\,
 \mathfrak X$, the morphism $$\BAut_{\mathfrak X} x\,\longrightarrow\,
 \BAut_{\mathcal G} x $$ is trivial by hypothesis, and the following lemma.

\begin{lemma}[\cite{alper:good_moduli} Remark 10.2]
Let $\mathcal F$ be a vector bundle on an algebraic stack $\mathfrak X$, and $x\,:\,
\spec \Omega \,\longrightarrow\,
 \mathfrak X$ be a geometric point with closed image. Then $\mathcal F_{|\mathcal G_x}$ is trivial if and only if $\BAut_{\mathfrak X} x$ acts trivially on 	$\mathcal F\bigotimes_{\mathcal O_{\mathfrak X,x}} \Omega$.
\end{lemma}

\begin{proof}
The last assertion means that $\mathcal F_{|\B\BAut_{\mathfrak X} x}$ is trivial. 
If $p_x\,:\,\mathcal G_x \,\longrightarrow\, \spec k(x)$ is the structure 
morphism, then $\mathcal F_{|\mathcal G_x}$ is trivial if and only if the morphism 
$${p_x}^*{p_x}_*\mathcal F_{|\mathcal G_x}\,\longrightarrow\,\mathcal F_{|\mathcal G_x}$$
is an isomorphism. Since this property is local, it can be checked on the cover
$\B\BAut_{\mathfrak X} x \,\longrightarrow\, \mathcal G_x$.
\end{proof}

Let us now specialize the previous discussion to neutral finite gerbes. We first recall a definition due to Joshua \cite[Definition 3.1 (i)]{joshua_RR}.

\begin{definition}
	\label{def:pres_inertia}
	A morphism of algebraic stacks $\mathfrak X' \,\longrightarrow\,
 \mathfrak X$ is \emph{isovariant} if the following diagram is Cartesian:
\[
\xymatrix{ I_{\mathfrak X'/ S} \ar[r]\ar[d] & I_{\mathfrak X/ S} \ar[d] \\
\mathfrak X' \ar[r] & \mathfrak X
}
\]
\end{definition}

\begin{remark}\mbox{}
	\begin{enumerate}
\item In \cite{noohi:fund_group_alg_stack}, the  alternative name ``fixed points
reflecting morphism" is used.

\item Monomorphisms of algebraic stacks are isovariant (Proposition 
\ref{rest_inert_stack}). However of course, there are many more examples, in 
particular any morphism between algebraic spaces is isovariant.

\item It is easy to see that the property of being isovariant is stable by base 
change, but is not local.
\end{enumerate}
\end{remark}

\begin{corollary}
\label{cor:tors-mod-space}
Let $G/S$ be a finite group scheme. With the hypothesis of Corollary
\ref{cor:gen_fg_moduli}, a $G$--torsor $\mathfrak X' \,\longrightarrow\,
 \mathfrak X$ descends to the moduli space if and only if it is isovariant.
\end{corollary}

\begin{proof}
According to 
Corollary \ref{cor:morph_fact_moduli}, the corresponding morphism $$\mathfrak X \,\longrightarrow\,
 \B G$$ factors through $\mathfrak X \,\longrightarrow\,
 Y$ if and only if it is trivial on inertia. But the sequence
$$\mathfrak X' \,\longrightarrow\,
 \mathfrak X \,\longrightarrow\,
 \B G$$ induces an exact sequence $$1\,\longrightarrow\,
 I_{\mathfrak X'/S} \,\longrightarrow\,
 \left(I_{\mathfrak X/S}\right)_{|\mathfrak X'} \,\longrightarrow\,
 \left(I_{\B G/S}\right)_{|\mathfrak X'}$$
and the result follows.
\end{proof}

\section{Nori-uniformization of stacks with finite inertia}\label{sec:unif}

In this section, we will restrict ourselves to \emph{algebraic stacks with finite inertia}, that is, the inertia stack $I_{\mathfrak X}\,\longrightarrow\,
 \mathfrak X$ is a finite group space. In particular tame stacks in the sense of \cite{aov:tame_stacks} are of this class.

\subsection{Nori-uniformizable stack} 

\begin{definition}
Let $\mathfrak X$ be a stack over a field $k$. We will say that $\mathfrak X$ is \emph{Nori-uniformizable} 
if there exists a representable $k$--morphism $\mathfrak X\,\longrightarrow\,
 \mathcal G$, where $\mathcal G/S$ is a finite gerbe.
\end{definition}

\begin{example}
	Assume that $k$ is of positive characteristic $p$ and put $$\mathfrak X \,= \,
[ {\mathbb P}^1 / \mu_p]$$ as in Example \ref{ex_gen}. Then $\mathfrak X$ is 
Nori-uniformizable, but it is not uniformizable by an algebraic space in the sense of \cite{beh-noo:unif}, since it is clear that the pro-\'etale fundamental gerbe $\pi^{et}_{\mathfrak X/S}$ is trivial.
\end{example}

Clearly, if there exists a finite $k$--group scheme $G$ and a $G$--torsor 
$X'\,\longrightarrow \,\mathfrak X$, where $X'$ is an algebraic space, then 
$\mathfrak X$ is Nori-uniformizable. As A. Vistoli indicated to us, it turns out 
that the converse is true. The key point is the following:

\begin{proposition}
	\label{prop:fin_gerb_alg}
Let $\mathcal G/S$ be an algebraic stack that is 
a fppf gerbe. Then $\mathcal G/S$ is smooth.
\end{proposition}

\begin{proof}
See \cite[Proposition A.2]{ber:functorial}.
\end{proof}

\begin{lemma}\label{lem:Weil_rest}
	Let $\mathcal G/S$ be a finite gerbe, and $k'/ k$ be a finite separable extension so that $\mathcal G(k')\,\neq\, \emptyset$. Denote by $\R_{k'/k}\cdot$ the Weil restriction along $k'/k$. Then 
	\begin{enumerate}
		\item $\R_{k'/k} \mathcal G_{k'}$ is a finite neutral gerbe over $k$.
		\item The canonical morphism $ \mathcal G\,\longrightarrow\,  \R_{k'/k} \mathcal G_{k'}$ is representable.
	\end{enumerate}
\end{lemma}

\begin{proof}
We fix a separable closure $k^{sep}$ of $k$. Then if $n\,=\,[k':k]$, we have:
	\begin{equation*}
	\label{eq:Weil_rest}
		\left(\R_{k'/k} \mathcal G_{k'}\right)_{k^{sep}}
\,\simeq\, \mathcal G_{k^{sep}}^{\times n}\; .
		\end{equation*}
       	\begin{enumerate}

\item From \cite[Lemma 6.2]{bv:nori_gerbe} we know that $\R_{k'/k} \mathcal G_{k'}$ is 
a finite stack. To prove that is it a finite gerbe, according to \cite[Proposition 
4.3]{bv:nori_gerbe}, it is enough to prove that it is geometrically connected and 
geometrically reduced. But if $\mathcal G_{k^{sep}}\,\simeq\, \B G$, it follows from 
the displayed formula that $\left(\R_{k'/k} \mathcal G_{k'}\right)_{k^{sep}} 
\,\simeq\, \B \left(G^{\times n} \right)$, hence $\left(\R_{k'/k} \mathcal 
G_{k'}\right)_{k^{sep}}$ is a gerbe, and so is geometrically connected and 
geometrically reduced. To conclude, by definition of Weil restriction, $\R_{k'/k} 
\mathcal G_{k'}(k)\,=\, \mathcal G_{k'}(k')\,\neq\, \emptyset$, that is, $\R_{k'/k} 
\mathcal G_{k'}$ is a neutral gerbe over $k$.
	  
\item Proposition \ref{char_rep}\eqref{char_rep_5} and the fact that being 
a monomorphism is local on the base for the fppf topology, \cite[Tag 
02YK]{stacks-project}, together show that being representable is also local on the base for 
the fppf topology. So it is enough to prove that $\mathcal 
G_{k^{sep}}\,\longrightarrow\, \left(\R_{k'/k} \mathcal G_{k'}\right)_{k^{sep}} $ 
is representable. But if $\mathcal G_{k^{sep}}\,\simeq\, \B G$, this morphism 
identifies with $\B G \,\longrightarrow\, \B \left(G^{\times n} \right)$, which is 
representable since the diagonal morphism $G\,\longrightarrow\, G^{\times n}$ is a 
monomorphism.
\end{enumerate}
\end{proof}

\begin{proposition}\label{prop:nori_unif_equiv}
Let $\mathfrak X$ be a stack over a field $k$. Then $\mathfrak X$ is 
Nori-uniformizable if and only if there exists a finite $k$--group scheme $G$ and 
a $G$--torsor $X'\,\longrightarrow \,\mathfrak X$, where $X'$ is an algebraic 
space.
\end{proposition}

\begin{proof}
It is enough to prove that any finite gerbe $\mathcal G/S$ has this last property. Since surjective and smooth morphisms have sections \'etale 
locally, it follows from Proposition \ref{prop:fin_gerb_alg} that there exists $k'/ 
k$ a finite separable extension so that $\mathcal G(k')\,\neq\, \emptyset$. Then 
according to Lemma \ref{lem:Weil_rest}, the canonical morphism
$ \mathcal G\,\longrightarrow\,  
\R_{k'/k} \mathcal G_{k'}$ is a representable morphism to a neutral gerbe.
\end{proof}

\begin{proposition}\label{unif_fund}
The Noetherian inflexible stack $\mathfrak X$ with finite inertia is Nori-uniformizable if and only if the
morphism $$\mathfrak X \,\longrightarrow\, \pi_{\mathfrak X/S}$$ is representable. 	
\end{proposition}

\begin{proof}
The ``only if'' part is clear. Indeed, if $\mathfrak X\,\longrightarrow\,
 \mathcal G$ is a representable morphism to a finite gerbe, it factors through the
morphism $\mathfrak X\,\longrightarrow\, \pi_{\mathfrak X/S}$, that must then be representable by Proposition \ref{char_rep}\eqref{char_rep_2}.

We will now prove the ``if'' part.
The morphism $\mathfrak X \,\longrightarrow \,\pi_{\mathfrak X/S}$ is the projective limit over
the directed set $D_{\mathfrak X}$ of all Nori-reduced morphisms $\mathfrak X \,\longrightarrow\,
\mathcal G$ to a finite gerbe (see \cite{bv:nori_gerbe}, proof of Theorem 5.7). It follows by commutation of limits that for relative inertia stacks
$$I_{\mathfrak X/ \pi_{\mathfrak X/S}}\,\simeq\, \varprojlim_{\mathfrak X
\rightarrow 
 \mathcal G} I_{\mathfrak X/ \mathcal G}\; . $$
By Proposition \ref{char_rep}\eqref{char_rep_3}, the assumption is equivalent to the fact that $I_{\mathfrak X/ \pi_{\mathfrak X/S}}$ is trivial as a group space over $\mathfrak X$. 
We have to prove that there exists a Nori-reduced morphism $$f_0\,:\,\mathfrak X
\,\longrightarrow\,
\mathcal G_0$$ such that $I_{\mathfrak X/ \mathcal G_0}$ is the trivial group space. 

More generally, we can consider, for any closed sub-stack $\mathfrak X'\,\subset
\,\mathfrak X$, the issue of finding such a morphism $f_0\,:\,\mathfrak X\,\longrightarrow\,
 \mathcal G_0$
satisfying the condition that $I_{\mathfrak X'/ \mathcal G_0}$ is  trivial. We proceed by Noetherian induction, and fix a closed
sub-stack $$\mathfrak X'\,\subset\, \mathfrak X\, ,$$ assuming that the problem has a solution for any strict closed
sub-stack $\mathfrak X''\,\subset\, \mathfrak X'$. Using the fact that $D_{\mathfrak X}$ is directed, we can suppose that $\mathfrak X'$ is irreducible. The same fact
shows that it is enough to prove that there exists a non-empty open sub-stack $\mathfrak U$ of $\mathfrak X'$ for which there exists $f_0\,:\,\mathfrak X\,\longrightarrow\,
 \mathcal G_0$ such that $I_{\mathfrak U/ \mathcal G_0}$ is trivial.

Let $f_1\,:\,\mathfrak X\,\longrightarrow\,
 \mathcal G_1$ be an arbitrary element of $D_{\mathfrak X}$. By
generic flatness (see Proposition \ref{inert_gen_flat}), there exists  a non-empty
open sub-stack $\mathfrak U_1$ of $\mathfrak X'$ such that
$I_{\mathfrak U_1/ \mathcal G_1}$ is flat. Being also finite, this group has a well
defined order. If this order is not $1$, as shown below, we can produce an
element $$f_2\,:\,\mathfrak X\,\longrightarrow\,
 \mathcal G_2$$ of $D_{\mathfrak X}$ and a non-empty
open sub-stack $\mathfrak U_2$ of $\mathfrak X'$ such that
$I_{\mathfrak U_2/ \mathcal G_2}$ is flat, and $\# I_{\mathfrak U_2/ \mathcal G_2} \,<\,
\# I_{\mathfrak U_1/ \mathcal G_1}$. This completes the proof of the proposition
by induction.

To prove the above claim, assume that $I_{\mathfrak U_1/ \mathcal G_1}$ is not
trivial. Since by assumption $\varprojlim_{\mathfrak X \rightarrow
 \mathcal G} I_{\mathfrak X/ \mathcal G}$ is trivial, there exists a morphism $f_2
\,:\,\mathfrak X
\,\longrightarrow\, \mathcal G_2$ mapping to $f_1$ in $D_{\mathfrak X}$ such that the induced monomorphism $I_{\mathfrak U_1/ \mathcal G_2}\,\longrightarrow\,
 I_{\mathfrak U_1/ \mathcal G_1}$ is not an isomorphism. Let $\mathfrak U_2$ be a nonempty subset of $\mathfrak U_1$ such that $I_{\mathfrak U_2/ \mathcal G_2}$ is flat. Since the cokernel of $I_{\mathfrak U_1/ \mathcal G_2}\,\longrightarrow\,
 I_{\mathfrak U_1/ \mathcal G_1}$, namely $f_2^* I_{\mathcal G_2/ \mathcal G_1}$, is flat, it remains non-trivial after restriction to $\mathfrak U_2$. Hence
we have $\# I_{\mathfrak U_2/ \mathcal G_2} \,<\, \# I_{\mathfrak U_1/ \mathcal G_1}$.
\end{proof}

 \begin{remark}
	We used mainly two aspects: the fact that $D_{\mathfrak X}$ is directed, and the fact that when $\mathfrak X$ is Noetherian and reduced, and $\mathfrak X\,\longrightarrow\,
 \mathcal G$ is a morphism, the relative inertia stack $I_{\mathfrak X/ \mathcal G}$ is flat over a non-empty open subset $\mathfrak U$ of $\mathfrak X$. This fact can be interpreted in the following way: $\mathfrak U$ is a gerbe over its coarse sheafification $\pi_0(\mathfrak U)$ over $\mathcal G$. When $\mathcal G\,=\,
S$, this is the core of the classical result called ``stratification by gerbes". For the relative version of this result, see appendix \ref{app_strat}.

Notice, however, that the flatness of the relative inertia stack is $I_{\mathfrak 
X/ \mathcal G}$ over a non-empty open subset does not follow from the flatness of 
the absolute inertia stack $I_{\mathfrak X/ \mathcal S}$, since the kernel of a 
morphism between two finite and flat group spaces is not necessarily flat. This 
is the main difference between our situation and the one considered in 
\cite{noohi:fund_group_alg_stack}. Since the kernel of a morphism between two 
finite and \'etale group spaces is finite and \'etale, Noohi can use directly 
stratification by gerbes over $S$.
\end{remark}

\begin{corollary}\label{unif_desc}
Let $k'/k$ be a finite separable extension. Then the stack $\mathfrak X/S$ is Nori-uniformizable if and only if $\mathfrak X_{k'}$ is Nori-uniformizable.
\end{corollary}

\begin{proof}
Since being representable by algebraic spaces is a local property, this follows from
Proposition \ref{unif_fund} and \cite[Proposition 6.1]{bv:nori_gerbe} (which
asserts that the fundamental gerbe commutes with finite
separable base change). 
\end{proof}

\subsection{Nori-uniformization and residual gerbes}

The following proposition generalizes Theorem 6.2 of
\cite{noohi:fund_group_alg_stack}.

\begin{proposition}\label{unif_res}
Let $\mathfrak X/S$ be an inflexible stack with finite inertia and of finite type. Then $\mathfrak X$ is
Nori-uniformizable if and only if for any closed point $x$, the canonical morphism
$$\mathcal G_x \,\longrightarrow \,\pi_{\mathfrak X /S}$$ is representable. 
\end{proposition}

\begin{proof}
This follows from Proposition \ref{unif_fund} and Lemma \ref{char_rep_closed} below. 
\end{proof}

\begin{lemma}\label{char_rep_closed}	
Let $\mathfrak X$ be a stack of finite type over a field $k$ and 
$f\,:\, \mathfrak X\,\longrightarrow\, \mathfrak Y$ be a morphism to an algebraic stack. Then $f$ is representable if and only if for any closed point $x\in \left|\mathfrak X \right|_0$, the induced morphism $\mathcal G_x\,\longrightarrow\, \mathfrak Y$ is representable.
\end{lemma}

\begin{proof}
By Proposition \ref{rest_inert_stack}, for any closed point $x\,:\,\spec \Omega \,
\longrightarrow \,\mathfrak X$, we have that $\left(I_{\mathfrak X/\mathfrak Y}\right)_{|{\mathcal G_x}}\simeq I_{\mathcal G_x/\mathfrak Y}$, hence the statement follows from Proposition \ref{char_rep}\eqref{char_rep_3} and the fact that the set of closed points is dense.
\end{proof}

We recall that, using the terminology of \cite{bv:nori_gerbe}, if $\mathfrak X$ is
pseudo-proper, then the pull-back along $\mathfrak X \,\longrightarrow\, \pi_{\mathfrak X/S}$ identifies representations of $\pi_{\mathfrak X/S}$ with the category $\EF(\mathfrak X)$ of essentially finite vector bundles on $\mathfrak X$. We can now state our main theorem.

\begin{theorem}
\label{char_unif}
Let $\mathfrak X/S$ be an inflexible and pseudo-proper stack with finite inertia and of finite type. Then $\mathfrak X$ is Nori-uniformizable if and only if for any closed point $x$, any representation $V$ of $\mathcal G_x$ is a subquotient of the restriction of an essentially finite vector bundle on $\mathfrak X$ along $\mathcal G_x\,\longrightarrow\, \mathfrak X$.
\end{theorem}

\textit{Proof.}~
According to Propositions \ref{unif_res} and \ref{char_rep}\eqref{char_rep_2}, the 
stack $\mathfrak X$ is Nori-uniformizable if and only if for any closed point $x$, 
the morphism $\mathcal G_x \,\longrightarrow \,\pi_{\mathfrak X/S} \bigotimes_k k(x)$ is 
representable. According to Proposition \ref{mor_tan}, this is equivalent to the fact 
that any representation $V$ of $\mathcal G_x$ is a subquotient of the restriction of 
a representation of $\pi_{\mathfrak X/S} \bigotimes_k k(x)$. Now the following lemma
completes the proof.

\begin{lemma}
Let $\mathcal G/S$ be a Tannakian gerbe, $k'/k$ an extension, and $f\,:\,
 \mathcal G_{k'}\,\longrightarrow\,
\mathcal G$ the canonical morphism. Then for any representation $V'$ of $\mathcal G_{k'}$, the canonical morphism $f^*f_* V\,\longrightarrow\,
 V$ is an epimorphism.
\end{lemma}

\begin{proof}
The morphism $f$ is affine, and in particular it is quasi-affine, and hence
the result follows (see \cite[Proposition 6.2]{alper-easton:recasting}).
\end{proof}

\appendix

\section{Representable morphisms}

We start by recalling a characterization of representable morphisms.

\begin{proposition}\label{char_rep}
	Let $f\,:\, \mathfrak X \,\longrightarrow\, \mathfrak Y$ be a morphism of $S$--stacks. The following properties are equivalent:

\begin{enumerate}
	\item The morphism $f$ is representable by algebraic spaces.

	\item \label{char_rep_2} For any section $\sigma \,:\, T\,\longrightarrow\, \mathfrak X$, 
the canonical morphism of $T$--group spaces $$\BAut_T \sigma \,\longrightarrow\, \BAut_T 
f(\sigma)$$ has trivial kernel.

	\item \label{char_rep_3} The relative inertia stack $I_{\mathfrak X/\mathfrak 
Y}\,=\,\mathfrak X\times_{\mathfrak X \times_{\mathfrak Y}\mathfrak X}\mathfrak X $ is 
trivial (as a group space over $\mathfrak X$).

	\item \label{char_rep_4} The group morphism $I_{\mathfrak X/S} \,\longrightarrow\,
f^*I_{\mathfrak Y/S}$ is a monomorphism.

	\item \label{char_rep_5} The diagonal $\Delta: \mathfrak X \,\longrightarrow\,
 \mathfrak X \times_{\mathfrak 
Y}\mathfrak X$ is a monomorphism.
\end{enumerate}
\end{proposition}

\begin{proof}
See \cite[Tag 04YY]{stacks-project} for
the equivalence of the first three statements. The fourth statement is equivalent to 
the third one, considering the following $2$--Cartesian diagram:

	\[
\xymatrix{ I_{\mathfrak X/\mathfrak Y} \ar[r]\ar[d] & I_{\mathfrak X/ S} \ar[d] \\
\mathfrak Y \ar[r] & I_{\mathfrak Y/ S}
}
\] 
The fifth statement is a reformulation of the third one; see Proposition \ref{mono_stack}.
\end{proof}

\section{Monomorphisms of algebraic stacks}

A morphism $f:\mathfrak X'\,\longrightarrow\,
 \mathfrak X$ of algebraic stacks is a monomorphism if it
is representable by a morphism of algebraic spaces that is a monomorphism (see
\cite[Tag 04ZV]{stacks-project} for details).

For convenience of the reader, we recall the following characterization.

\begin{proposition}
\label{mono_stack}
 Let $f\,:\,\mathfrak X'\,\longrightarrow\,
 \mathfrak X$ be a morphism of algebraic stacks. The following are equivalent:
\begin{enumerate}
	\item $f$ is a monomorphism,
	\item $f$ is fully faithful,
	\item the diagonal $\Delta_f\,:\, \mathfrak X'\,\longrightarrow\,
 \mathfrak X'\times _{\mathfrak X}\mathfrak X'$ is an isomorphism.
\end{enumerate}
\end{proposition}

\begin{proof}
See \cite[Tag 04ZZ]{stacks-project}.
\end{proof}

\begin{proposition}
	\label{rest_inert_stack}
	Let $\mathcal S/S$ be a base stack, and let $\mathfrak X'\,\longrightarrow\,
 \mathfrak X$ be a
$\mathcal S$--monomorphism of $\mathcal S$--algebraic stacks. Then the following diagram
is $2$--Cartesian:

\[
\xymatrix{ I_{\mathfrak X'/\mathcal S} \ar[r]\ar[d] & I_{\mathfrak X/\mathcal S} \ar[d] \\
\mathfrak X' \ar[r] & \mathfrak X
}
\] 
\end{proposition}

\begin{proof}
	This follows from the absolute statement ($\mathcal S\,=\,S$,
\cite[Tag 06R5]{stacks-project}) and the following $2$--Cartesian diagram: 
	\[
\xymatrix{ I_{\mathfrak X/\mathcal S} \ar[r]\ar[d] & I_{\mathfrak X/ S} \ar[d] \\
\mathcal S \ar[r] & I_{\mathcal S/ S}
}
\] 
\end{proof}

\begin{remark}
With the terminology introduced in Definition \ref{def:pres_inertia}, Proposition \ref{rest_inert_stack} means exactly that monomorphisms are isovariant.	
\end{remark}

\section{Stratification by gerbes over a base stack}
\label{app_strat}

Let $f\,:\,\mathfrak X\,\longrightarrow\,
 \mathfrak Y$ be a morphism of algebraic stacks over some base $S$. We assume for simplicity that the diagonal $\Delta_f \,:\, \mathfrak X \,\longrightarrow\,
 \mathfrak X \times _{\mathfrak Y} \mathfrak X$ is quasi-compact (equivalently, it is of finite type). Then the relative inertia stack $I_{\mathfrak X /\mathfrak Y}\,\longrightarrow\,
 \mathfrak X$ is a group space of finite type and, if we further assume that $\mathfrak X$ is
Noetherian and reduced, then we can apply the classical generic flatness
theorem, \cite[Th\'eor\`eme 6.9.1]{EGAIV2}, to get the following.

\begin{proposition}
	\label{inert_gen_flat}
	Let $f\,:\,\mathfrak X\,\longrightarrow\,
 \mathfrak Y$ be a morphism of algebraic stacks with quasi-compact diagonal, and assume that $\mathfrak X$ is Noetherian and reduced. Then there exists a dense open sub-stack $\mathcal U\subset \mathcal X$ such that $I_{\mathfrak U /\mathfrak Y}\,\longrightarrow\,
 \mathfrak U$ is flat.
\end{proposition}

\begin{proof}
See \cite[Tag 06RC]{stacks-project}, for the absolute version.
\end{proof}

The flatness of the inertia stack has a standard interpretation. We start by 
giving the definition of an ``absolute'' gerbe in this relative setting.

\begin{definition}
Let $f\,:\,\mathfrak X \,\longrightarrow\,
 \mathfrak Y$ be a morphism of algebraic stacks. We say
that \emph{$\mathfrak X$ is a gerbe in $\mathfrak Y$--stacks}
if there exists a factorization $\mathfrak X \,\longrightarrow\,
\mathfrak Z \,\longrightarrow\,\mathfrak Y$ of
$f$ such that $\mathfrak X \,\longrightarrow\,
\mathfrak Z$ is a gerbe, and $\mathfrak Z \,\longrightarrow\,
\mathfrak Y$ is representable by algebraic spaces.  
\end{definition}

\begin{remark}\mbox{}
	\begin{enumerate}
\item This definition is the direct generalization of the absolute version given
in \cite[Tag 06QC]{stacks-project}.

\item The condition that $\mathfrak X \,\longrightarrow\,
 \mathfrak Z$ is a gerbe
roughly means that $\mathfrak X$ is a gerbe if we endow $\mathfrak Z$ from the topology inherited from the base
$S$; see \cite[Tag 06P2]{stacks-project} for details.

\item The stack $\mathfrak Z$, if it exists, is
unique, and it is obtained by sheafifying, over $\mathfrak Y$ endowed with its topology inherited from the base $S$, the presheaf $U \,\longmapsto\, \Ob(\mathcal{X}_U)/\!\!\cong$ (see \cite[Tag 06QD]{stacks-project}).
\end{enumerate}
\end{remark}

\begin{proposition}
The stack $\mathfrak X$ is a gerbe in $\mathfrak Y$--stacks if and only if 
$I_{\mathfrak X /\mathfrak Y}\,\longrightarrow\,
\mathfrak X$ is flat and locally of finite presentation.
\end{proposition}

\begin{proof}
See \cite[Tag 06QJ]{stacks-project}, for the absolute version.
\end{proof}

From Proposition \ref{inert_gen_flat} we have the following:

\begin{theorem}
Let $f\,:\,\mathfrak X\,\longrightarrow\,
\mathfrak Y$ be a morphism of algebraic stacks with
quasi-compact diagonal, and assume that $\mathfrak X$ is Noetherian. Then there exists a finite decomposition $\mathfrak X\,=\,\coprod_{i\in I}\mathfrak X_i$ of $\mathfrak X$ by locally closed sub-stacks such that, for all $i\in I$, the stack $\mathfrak X_i$, endowed with the reduced structure, is a gerbe in $\mathfrak Y$--stacks.	
\end{theorem}

\begin{remark}
Our formulation of the statement, based on the classical generic flatness
theorem (see \cite[Th\'eor\`eme 6.9.1]{EGAIV2}) is rather restrictive, even if it is more than enough for our purposes (in fact we only need Proposition \ref{inert_gen_flat}). For a more general version, based on a more powerful generic flatness theorem, see \cite[Tag 06RF]{stacks-project}.
\end{remark}

\section{Morphisms of Tannakian gerbes}

The following proposition is well known for \emph{neutral} gerbes (see
\cite[II 4.3.2]{rivano:cat_tan}); it is included here due to the lack of
a reference for the more general statement.

\begin{proposition}
	\label{mor_tan}
 Let $\phi \,:\, \mathcal G \,\longrightarrow \,\mathcal G'$ be a morphism between Tannakian
gerbes, and let $\phi ^*\,:\, \Vect \mathcal G' \,\longrightarrow\, \Vect \mathcal G$ be
the corresponding Tannakian functor.
 \begin{enumerate}
\item The morphism $\phi$ is representable if and only if any object of $\Vect 
\mathcal G$ is a subquotient of the image by $\phi ^*$ of an object of $\Vect 
\mathcal G'$.
	
\item The morphism $\phi$ is a (relative) gerbe if and only the functor $\phi ^*$ 
is fully faithful, and the essential image of $\phi^*$ is stable by subobject.
\end{enumerate}
\end{proposition}

\begin{proof}\mbox{}
\begin{enumerate}
\item We recall that given a base $S$, there is a canonical morphism $\bd \,:\, \Gr_S\,\longrightarrow\,
 \Bd_S$ from the stack of groups to the stack of bands. A group morphism $\Phi\,:\,G 
\,\longrightarrow\,G'$ is a monomorphism if and only if the corresponding band morphism 
$$\bd(\Phi)\,:\,\bd(G)\,\longrightarrow\,
\bd(G')$$ is injective (indeed by definition a morphism of bands is injective if it is locally represented by a group monomorphism).
 
Moreover each gerbe $\varphi\,:\,\mathcal G \,\longrightarrow\, S$ admits a well 
defined $S$--band $\bd(\mathcal G)$, and there is a natural isomorphism 
$\varphi^*(\bd(\mathcal G))\,\simeq\, \bd(I_{\mathcal G})$. To check this, we recall 
that the association $\mathcal G \,\longmapsto \,\bd(\mathcal G)$ is characterized 
by three properties: it is functorial, compatible with localization, and $\bd(\B 
G)=\bd(G)$. But when we base change $\mathcal G \,\longrightarrow\, S$ by itself, 
it is easy to check that we get the neutral gerbe $\B I_{\mathcal G} 
\,\longrightarrow\, \mathcal G$.

According  to \cite[III 3.3.3]{rivano:cat_tan}, any object of $\Vect \mathcal G$ is a
subquotient of the image by $\phi^*$ of an object of $\Vect \mathcal G'$ if and only if
the morphism $$\bd(\phi)\,:\,\bd(\mathcal G)\,\longrightarrow\,
\bd(\mathcal G')$$ is injective. 

Since the structural morphism $\varphi\,:\,\mathcal G
\,\longrightarrow\,
S$ is a covering, this is equivalent to the
assertion that $\varphi^*\bd(\phi)\,:\,\varphi^*\bd(\mathcal G)\,\longrightarrow\,
\varphi^*\bd(\mathcal G')$ is injective; in other words, equivalent to the assertion that the natural morphism
$\bd(I_{\mathcal G})\,\longrightarrow\,
 \bd(\phi^* I_{\mathcal G'})$ is injective. This is in turn
equivalent to the fact that the morphism $I_{\mathcal G}\,\longrightarrow\,
 \phi^* I_{\mathcal G'}$ is a
monomorphism, and we conclude by Proposition \ref{char_rep}(\ref{char_rep_4}).
	\item Since this is similar to the proof of the first part, we omit
the details. Also, this is not used in the present article.
\end{enumerate}
This completes the proof.
\end{proof}

\section*{Acknowledgements}

We thank Angelo Vistoli for very useful advice and especially for pointing out 
\cite{alper:good_moduli} to us. Proposition \ref{unif_fund} and its proof are due 
to him. We thank the two referees for detailed comments that helped in improving 
the exposition. The first-named author wishes to thank Universit\'e Lille 1 for 
its hospitality.

\end{document}